\documentclass{amsart}
\usepackage{bbm}
\usepackage{hyperref}

\newcommand{\R}{\mathbbm R}

\newcommand{\N}{\mathbbm N}

\newtheorem{theorem}{Theorem}

\newtheorem{lemma}[theorem]{Lemma}

\theoremstyle{definition}

\theoremstyle{remark}

\numberwithin{equation}{section}
\numberwithin{theorem}{section}

\newcommand{\coloneqq}{\,\raise0.08ex\hbox{\textnormal{:}}\!\!=}

\def\XXint#1#2#3{{\setbox0=\hbox{$#1{#2#3}{\int}$}
     \vcenter{\hbox{$#2#3$}}\kern-.5\wd0}}

\begin{document}

\title[Well-posedness of the supercritical Lane-Emden heat flow]
{Well-posedness of the supercritical Lane-Emden heat flow in Morrey spaces}
\author{Simon Blatt}
\thanks{S. B. gratefully acknowledges the support of the Forschungsinstitut f\"ur 
Mathematik (FIM), ETH Zurich}
% and of the Swiss National Science Foundation under grant no. 200020-125127.}
\address[Simon Blatt] {Fachbereich Mathematik, Universit\"at Salzburg, Hellbrunner Str. 34, 
A-5020 Salzburg}
\email{simon.blatt@sbg.ac.at}
\author{Michael Struwe}
\address[Michael Struwe]{Departement Mathematik\\ETH-Z\"urich\\CH-8092 Z\"urich}
\email{michael.struwe@math.ethz.ch}

\date{\today}

\begin{abstract}
For any smoothly bounded domain $\Omega\subset\R^n$, $n\ge 3$, and any exponent
$p>2^*=2n/(n-2)$ we study the Lane-Emden heat flow $u_t-\Delta u = |u|^{p-2}u$ on 
$\Omega\times]0,\infty[$ and establish local and global well-posedness results 
for the initial value problem with suitably small initial data $u\big|_{t=0}=u_0$ in the 
Morrey space $L^{2,\lambda}(\Omega)$, where $\lambda=4/(p-2)$. We contrast our
results with results on instantaneous complete blow-up of the flow for 
certain large data in this space, similar to ill-posedness results of 
Galaktionov-Vazquez for the Lane-Emden flow on $\R^n$. 
\end{abstract}

\maketitle

\section{Introduction} 
Let $\Omega$ be a smoothly bounded domain in $\R^n$, $n\ge 3$, and let $T>0$.
Given initial data $u_0$, we consider the Lane-Emden heat flow
\begin{equation}\label{1.1}
  u_t -\Delta u = |u|^{p-2}u \hbox{ on } \Omega\times [0,T[,\ 
  u=0 \hbox{ on } \partial\Omega\times [0,T[,\ u\big|_{t=0}=u_0
\end{equation}
for a given exponent $p>2^*=2n/(n-2)$, that is, in the ``supercritical'' regime. 

As observed by Matano-Merle \cite{Matano-Merle-2009}, p. 1048, 
the initial value problem \eqref{1.1} may be ill-posed for certain
data $u_0\in H_0^1\cap L^p(\Omega)$; see also our results in Section 4 below.
However, as we had shown in two previous papers 
\cite{Blatt-Struwe-2013}, Section 6.5, \cite{Blatt-Struwe-2014}, Remark 3.3,
the Cauchy problem \eqref{1.1} is globally
well-posed for suitably small data $u_0$ belonging to the Morrey space 
$H_0^{1,\mu}\cap L^{p,\mu}(\Omega)$, where $\mu=\frac{2p}{p-2}<n$. 
Here we go one step further and show that problem \eqref{1.1} even is  
well-posed for suitably small data $u_0\in L^{2,\lambda}(\Omega)\supset L^{p,\mu}(\Omega)$, 
where $\lambda=\frac{2\mu}{p}=\frac{4}{p-2}=\mu-2$, thus considerably improving on
the results of Brezis-Cazenave \cite{Brezis-Cazenave-1996} or Weissler \cite{Weissler-1981}
for initial data in $L^q$, $q\ge n(p-2)/2$. Our results are similar to 
results of Taylor \cite{Taylor-1992} who demonstrated local and global well-posedness
of the Cauchy problem for the equation 
\begin{equation*}
  u_t -\Delta u = DQ(u) \hbox{ on } \Omega\times [0,T[,
\end{equation*}
for suitably small initial data $u\big|_{t=0}=u_0$ in a Morrey space, 
where $D$ is a linear differential operator of first order and $Q$ is a quadratic form in $u$
as in the Navier-Stokes system. However, similar to the work of Koch-Tataru \cite{Koch-Tataru-2001}
on the Navier-Stokes system, in our treatment of \eqref{1.1} we are able to completely 
avoid the use of pseudodifferential operators in favor of simple integration by parts.

The study of the initial value problem for \eqref{1.1} for non-smooth
initial data is motivated by the question whether a solution $u$ of \eqref{1.1}
blowing up at some time $T<\infty$ can be extended as a weak solution of \eqref{1.1} 
on a time interval $]0,T_1[$ for some $T_1>T$. Note that if such a continuation is possible
and if the extended solution still satisfies the monotonicity formula \cite{Blatt-Struwe-2014},
Proposition 3.1, it follows that $u(T)\in L^{2,\lambda}(\Omega)$. Hence, the regularity 
assumption $u_0\in L^{2,\lambda}(\Omega)$ is necessary from this point of view 
and cannot be weakened. However, our results in Section 4 
show that the condition $u(T)\in L^{2,\lambda}(\Omega)$ in general is not sufficient for
continuation and that a smallness condition as in our Theorems \ref{thm1.1}, \ref{thm1.2}
below is needed.

Note that the question of continuation after blow-up only is of relevance in the 
supercritical case when $p>2^*$. 
Indeed, by work of Baras-Cohen \cite{Baras-Cohen-1987} in the subcritical case 
$p<2^*$ a classical solution $u\ge 0$ to \eqref{1.1} blowing up at 
some time $T<\infty$ always undergoes ``complete blow-up'' (see Section 4 for a definition),
and $u$ cannot be continued as a (weak) solution to \eqref{1.1} after time $T$ in any reasonable 
way. In \cite{Galaktionov-Vazquez-1997} Galaktionov und Vazquez extend the Baras-Cohen result 
to the critical case $p=2^*$.

In the next section we state our well-posedness results, which we prove in Section 3. 
In Section 4 we then contrast these results with results on instantaneous complete 
blow-up of the flow for certain large data. These results are similar to ill-posedness 
results of Galaktionov-Vazquez for the Lane-Emden flow on $\R^n$; see for instance 
\cite{Galaktionov-Vazquez-1997}, Theorem 10.4. We conclude the paper 
with some open problems.

\section{Global and local well-posedness}
Recall that for any $1\le p<\infty$, $0<\lambda<n$ (in Adams' \cite{Adams-75} notation)
a function $f\in L^p(\Omega)$ on a domain $\Omega \subset \R^n$ belongs to the Morrey space 
$L^{p,\lambda}(\Omega)$ if
\begin{equation}\label{1.2}
 \|f\|^p_{L^{p,\lambda}(\Omega)} := \sup_{x_0 \in R^n,\; r>0} r^{\lambda-n} 
\int_{B_r(x_0) \cap \Omega} |f|^p dx <\infty, 
\end{equation}
where $B_r(x_0)$ denotes the Euclidean ball of radius $r>0$ centered at $x_0$.
Moreover, we write $f\in L_0^{p,\lambda}(\Omega)$ whenever $f\in L^{p,\lambda}(\Omega)$
satisfies 
\begin{equation*}
\sup_{x_0\in R^n,\;0<r<r_0} r^{\lambda-n} 
\int_{B_r(x_0)\cap \Omega}|f|^p dx\to 0 \hbox{ as } r_0\downarrow 0.
\end{equation*}

Similarly, for any $1\le p<\infty$, $0<\mu<n+2$ a function $f\in L^p(E)$ on 
$E\subset\R^n\times\R$ belongs to the parabolic Morrey space $L^{p,\mu}(E)$ if
\begin{equation*}
 \|f\|^p_{L^{p,\mu}(E)} := \sup_{z_0=(x_0,t_0) \in R^{n+1}, r>0} r^{\mu-(n+2)} 
\int_{P_r(z_0) \cap E} |f|^p dz < \infty,
\end{equation*}
where $P_r(x,t)$ denotes the backwards parabolic cylinder $P_r(x,t)=B_r(x)\times ]t-r^2,t[$. 

Given $p>2^*$, we now fix the Morrey exponents $\mu=\frac{2p}{p-2}$ and 
$\lambda=\frac{4}{p-2}=\mu-2$, which are natural for the study of problem \eqref{1.1}. 

Throughout the following a function $u$ will be called a smooth solution of \eqref{1.1} on $]0,T[$
if $u\in C^1(\bar\Omega\times]0,T[)$ with $u_t\in L^2_{loc}(\bar{\Omega}\times[0,T[)$ solves 
\eqref{1.1} in the sense of distributions and achieves the initial data in the sense of traces. 
By standard regularity theory then $u$ also is of class $C^2$ with respect to $x$ and 
satisfies \eqref{1.1} classically. Schauder theory, finally, yields even higher regularity
to the extent allowed by smoothness of the nonlinearity $g(v)=|v|^{p-2}v$. The function $u$ will 
be called a global smooth solution of \eqref{1.1} if the above holds with $T=\infty$. 

Our results on local and global well-posedness are summarized in the following theorems.

\begin{theorem}\label{thm1.1}
Let $\Omega\subset\R^n$ be a smoothly bounded domain, $n\ge 3$. 
There exists a constant $\varepsilon_0>0$ such that for any function 
$u_0\in L^{2,\lambda}(\Omega)$ satisfying $\|u_0\|_{L^{2,\lambda}}<\varepsilon_0$ there 
is a unique global smooth solution $u$ to \eqref{1.1} on $\Omega\times]0,\infty[$. 
\end{theorem}

The smallness condition can be somewhat relaxed.

\begin{theorem}\label{thm1.2}
Let $u_0\in L^{2,\lambda}(\Omega)$ and suppose that there exists a number $R>0$ such that
\begin{equation*}
  \sup_{x_0\in R^n,\;0<r<R}r^{\lambda-n}\int_{B_r(x_0)\cap\Omega}|u_0|^2 dx
  \le\varepsilon_0^2,
\end{equation*}
where $\varepsilon_0>0$ is as determined in Theorem \ref{thm1.1}.
Then there exists a unique smooth solution $u$ to \eqref{1.1} on an interval $]0,T_0[$,
where $T_0/R^2=C(\varepsilon_0/\|u_0\|_{L^{2,\lambda}})>0$.

In particular, for any $u_0\in L_0^{2,\lambda}(\Omega)$ there exists a unique smooth solution 
$u$ to \eqref{1.1} on some interval $]0,T[$, where $T=T(u_0)>0$.
\end{theorem}

It is well-known that for smooth initial data $u_0\in C^1(\bar\Omega)$ there exists a smooth 
solution $u$ to the Cauchy problem \eqref{1.1} on some time interval $[0,T[$, $T>0$. 
By the uniqueness of the solution to \eqref{1.1} constructed in Theorem \ref{thm1.1}
or \ref{thm1.2}, the latter solution coincides with $u$ and hence is smooth up to $t=0$ 
if $u_0\in C^1(\bar\Omega)$.

\section{Proof of Theorem \ref{thm1.1}}
Let $n\ge 3$ and let
$$G(x,t)=(4\pi t)^{-n/2}e^{-\frac{|x|^2}{4t}},\ x\in\R^n,t>0,$$
be the fundamental solution to the heat equation on $\R^n$ with singularity at $(0,0)$.
Given a domain $\Omega\subset\R^n$ also let $\Gamma=\Gamma(x,y,t)=\Gamma(y,x,t)$ be the 
corresponding fundamental solution to the heat equation on $\Omega$ with homogeneous 
Dirichlet boundary data $\Gamma(x,y,t)=0$ for $x\in\partial\Omega$. 
Note that by the maximum principle for any $x,y\in\Omega$, any $t>0$ 
there holds $0<\Gamma(x,y,t)\le G(x-y,t)$.

For $x\in\Omega$, $r>0$ we let 
$$\Omega_r(x)=B_r(x)\cap\Omega;$$
similarly, for $x\in\Omega$, $r,t>0$ we define 
$$Q_r(x,t)=P_r(x,t)\cap\Omega\times ]0,\infty[.$$
We sometimes write $z=(x,t)$ for a generic point in space-time.
The letter $C$ will denote a generic constant, sometimes numbered for clarity. 

For $f \in L^1(\Omega)$ set
\begin{equation*}
 (S_{\Omega}f)(x,t) := \int_{\Omega}\Gamma(x,y,t)f(y)\, dy,\ t>0,
\end{equation*}
so that $v=S_{\Omega}f$ solves the equation
\begin{equation}\label{2.1}
  v_t -\Delta v = 0 \text{ on } \Omega\times [0,\infty[
\end{equation}
with boundary data $v(x,t)=0$ for $x\in\partial\Omega$ and initial data 
$v\big|_{t=0}=f$ on $\Omega$.

Similar to \cite{Blatt-Struwe-2013}, Proposition 4.3, by adapting the methods of 
Adams \cite{Adams-75} we can show that $S_{\Omega}$ is well-behaved on Morrey spaces. 
Recall that $\mu=\frac{2p}{p-2}$ with $2<\mu<n$, and $\lambda=\mu-2=\frac{4}{p-2}>0$.

\begin{lemma} \label{lem2.1}
i) For any $p>2^*=\frac{2n}{n-2}$ the map
\begin{equation*}
  S_{\Omega}\colon L^{2,\lambda}(\Omega)\ni f \mapsto (v,\nabla v)
  \in L^{p,\mu}\times L^{2,\mu}(\Omega\times [0,\infty[)
\end{equation*}
is well-defined and bounded. Moreover, we have the bounds
\begin{equation}\label{2.2}
  \|v(t)\|^2_{L^{\infty}} %+t\|\nabla v(t)\|^2_{L^{\infty}}
  \le Ct^{-\lambda/2}\|f\|^2_{L^{2,\lambda}},\
  \|v(t)\|^2_{L^{2,\lambda}}\le C\|f\|^2_{L^{2,\lambda}},\ t>0.
\end{equation}
ii) Let $f\in L^{2,\lambda}(\Omega)$ and suppose that for a given $\varepsilon_0>0$
there exists a number $R>0$ such that
\begin{equation*}
  \sup_{x_0\in\Omega,\;0<r<R}\Big(r^{\lambda-n}\int_{\Omega_r(x_0)}|f|^2 dx\Big)^{1/2}
  \le\varepsilon_0.
\end{equation*}
Then with a constant $C>0$ for $v=S_{\Omega}f$ there holds the estimate
\begin{equation*}
   \sup_{x_0\in\Omega,\;0<r^2\le t_0\le T_0}\Big(r^{\mu-n-2}\int_{Q_r(x_0,t_0)}|v|^p dz\Big)^{1/p}
   \le C\varepsilon_0,
   %+ \sup_{x_0\in\Omega,\;0<r^2\le t_0\le T}\Big(r^{\mu-n-2}
   %\int_{Q_r(x_0,t_0)}|\nabla v|^2 dz\Big)^{1/2}\le C\varepsilon_0,
\end{equation*}
where $T_0/R^2=C(\varepsilon_0/\|f\|_{L^{2,\lambda}(\Omega)})>0$.
\end{lemma}

\begin{proof}
i) Let $f\in L^{2,\lambda}(\R^n)$ and set $v=S_{\Omega}f$ as above. 
Recall the definition of the fractional maximal functions
\begin{equation*}
  M_\alpha f(x):= \sup_{r>0}M_{\alpha,r}f(x),\
  M_{\alpha,r}f(x):=r^{\alpha-n} \int_{\Omega_r (x)} |f(y)|\, dy,\ \alpha >0.
\end{equation*}
Note that H\"older's inequality gives the uniform bound
\begin{equation}\label{2.3}
  \big(M_{\lambda/2} f\big)^2 \leq M_{\lambda} (|f|^2) \leq \|f\|^2_{L^{2,\lambda}}.
\end{equation}

Following the scheme outlined by Adams \cite{Adams-75}, proof of Proposition 3.1, we
first derive pointwise estimates for $v$ %and $\nabla v$ 
and bounds on parabolic cylinders $P_r(x_0,t_0)$ with radius $r$ satisfying $0<2r^2<t_0$.
Using the well known estimate
\begin{equation*}
  G(x-y,t)\le C(|x-y|+\sqrt{t})^{-n} 
\end{equation*}
for the heat kernel and recalling that $\Gamma(x,y,t)\le G(x-y,t)$, for any $t>0$ we can bound
\begin{align*}
  |v(x,t)| & \le C\int_{\Omega}(|x-y|+\sqrt{t})^{-n}|f(y)|\,dy\\
  &\le C\int_{\Omega_{\sqrt{t}}(x)}(|x-y|+\sqrt{t})^{-n}|f(y)|\,dy\\
  &\qquad+C\sum_{k=1}^\infty\int_{\Omega_{2^k\sqrt{t}}(x)\setminus \Omega_{2^{k-1}\sqrt{t}}(x)}
  (|x-y|+\sqrt{t})^{-n}|f(y)|\,dy\\
  &\le C \sum_{k=0}^\infty (2^k\sqrt{t})^{-n} (2^k\sqrt{t})^{n-\lambda/2}
  M_{\lambda/2,2^k\sqrt{t}}f(x)\le C t^{-\lambda/4}M_{\lambda/2}f(x).
\end{align*}
Hence by \eqref{2.3} with a uniform constant $C>0$ for any $t>0$ there holds
\begin{equation*}
  \|v(t)\|^2_{L^{\infty}}%+t\|\nabla v(t)\|^2_{L^{\infty}}
  \le Ct^{-\lambda/2}\|M_{\lambda/2} f\|^2_{L^{\infty}}  
  \le Ct^{-\lambda/2}\|f\|^2_{L^{2,\lambda}},
\end{equation*}
as claimed in \eqref{2.2}. Moreover, for any $x_0\in\R^n$, any $t_0>0$ and any $0<r<\sqrt{t_0/2}$ we obtain 
the bounds
\begin{equation}\label{2.4}
  \|v(t_0)\|^2_{L^2(\Omega_r(x_0))}\le Cr^nt_0^{-\lambda/2}\|f\|^2_{L^{2,\lambda}}
  \le Cr^{n-\lambda}\|f\|^2_{L^{2,\lambda}},
\end{equation}
and similarly
\begin{equation}\label{2.5}
  \|v\|^p_{L^{p}(Q_r(x_0,t_0))}\le Cr^{n+2}t_0^{-p\lambda/4}\|f\|^p_{L^{2,\lambda}}
  \le Cr^{n+2-\mu}\|f\|^p_{L^{2,\lambda}},
\end{equation}
where we also used that $\mu=2p\lambda/4$.

In order to derive \eqref{2.5} also for radii $r\ge\sqrt{t_0/2}$ we need to 
argue slightly differently. We may assume that $x_0=0$. Moreover, after enlarging $t_0$, 
if necessary, we may assume that $t_0=2r^2$. 
Let $\psi=\psi_0=\psi_0(x)$ be a smooth cut-off function satisfying 
$\chi_{B_{r}(0)}\le\psi\le\chi_{B_{2r}(0)}$ and with $|\nabla\psi|^2\le 4r^{-2}$.
Set $r=:r_0$ and let $r_i=2^ir_0$, $\psi_i(x)=\psi(2^{-i}x)$, $i\in\N$. 
For ease of notation in the following estimates we drop the index $i$. 

Upon multipying \eqref{2.1} with $v\psi^2$ we find the equation
\begin{equation*}
 \begin{split}
  \frac12\frac{d}{dt}&(|v|^2\psi^2)-div(v\nabla v\psi^2)+|\nabla v|^2\psi^2\\
  &=-2v\nabla v\psi\nabla\psi\le\frac12|\nabla v|^2\psi^2+ 2|v|^2|\nabla\psi|^2.
  \end{split}
\end{equation*}
Integrating over $\Omega\times ]0,t_1[$ and using the bound $|\nabla\psi|^2\le 4r^{-2}$,
for any $0<t_1<t_0$ we obtain 
\begin{equation}\label{2.6}
 \begin{split}
  \int_{\Omega_{2r}(0)}&|v(t_1)|^2\psi^2dx
   +\int_{\Omega_{2r}(0)\times ]0,t_1[}|\nabla v|^2\psi^2dxdt \\
  &\le \int_{\Omega_{2r}(0)}|f|^2\psi^2dx
   +16r^{-2}\int_{\Omega_{2r}(0)\times ]0,t_1[}|v|^2dxdt.
  \end{split}
\end{equation}
For $r=r_i$, $i\in\N_0$, set 
\begin{equation*}
     \Psi(r):=\sup_{x_0\in\Omega,0<t<t_0}r^{\lambda-n}\int_{\Omega_{r}(x_0)}|v(t)|^2dx.
\end{equation*}
Recalling that $\lambda=\mu-2$, then from the previous inequality \eqref{2.6}
with the uniform constants $C_1=2^{n-\lambda}$, $C_2=32C_1$ we obtain
\begin{equation*}
 \begin{split}
     \Psi(r_i) &\le r_i^{\lambda-n}\Big(\int_{\Omega_{2r_i}(0)}|f|^2dx
     +16t_0r_i^{-2}\sup_{0<t<t_0}\int_{\Omega_{2r_i}(0)}|v(t)|^2dx\Big)\\
     &\le C_1\|f\|^2_{L^{2,\lambda}}+C_22^{-2i}\Psi(r_{i+1}).
  \end{split}
\end{equation*}
By iteration, for any $k_0\in\N$ there results
\begin{equation*}
 \begin{split}
     \Psi(r_0) &\le C_1\|f\|^2_{L^{2,\lambda}}+C_2\Psi(r_1)
     \le C_1(1+C_2)\|f\|^2_{L^{2,\lambda}}+C_2^22^{-2}\Psi(r_2)\le \dots\\
     &\le C_1\sum_{k=0}^{k_0}C_2^k2^{(1-k)k}\|f\|^2_{L^{2,\lambda}}
     +C_2^{k_0+1}2^{-k_0(k_0+1)}\Psi(r_{k_0+1}).
  \end{split}
\end{equation*}
Passing to the limit $k_0\to\infty$, we obtain that $\Psi(r_1)\le C\|f\|^2_{L^{2,\lambda}}$. 
Inserting this information into \eqref{2.6}, where we again set $r=r_0$, then we find
\begin{equation}\label{2.7}
     \Psi(r) +\sup_{x_0\in\Omega}r^{\mu-2-n}\int_{\Omega_r(x_0)\times]0,t_0[}|\nabla v|^2dxdt
     \le C\|f\|^2_{L^{2,\lambda}}.
\end{equation}
In particular, together with \eqref{2.4} we have now shown the bound
\begin{equation*}
\|v(t)\|^2_{L^{2,\lambda}}\le C\|f\|^2_{L^{2,\lambda}} \hbox{  for all }t>0,
\end{equation*}
and thus have verified \eqref{2.2} completely.

To complete the proof of \eqref{2.5} for $r=r_0=\sqrt{t_0/2}$, let $\psi=\psi_0$ 
as above and let $\tau(t)=\min\{t,t_0-t\}$. 
Multiplying \eqref{1.1} with the function $v|v|^{p-2}\psi^2\tau$ then we obtain 
\begin{equation*}
 \begin{split}
  \frac1p\frac{d}{dt}&(|v|^p\psi^2\tau)-\frac1p\frac{d\tau}{dt}|v|^p\psi^2
  -div(|v|^{p-2}v\nabla v\psi^2\tau)+(p-1)|\nabla v|^2|v|^{p-2}\psi^2\tau\\
  &=-2|v|^{p-2}v\nabla v\psi\nabla\psi\tau
  \ge -|\nabla v|^2|v|^{p-2}\psi^2\tau-|v|^p|\nabla\psi|^2\tau.
  \end{split}
\end{equation*}
Integrating over $\Omega\times ]0,t_0[$ and using that $\frac{d\tau}{dt}=1$ for $0<t<t_0/2$,
$\frac{d\tau}{dt}=-1$ for $t_0/2<t<t_0$, as well as the fact that the region 
$\Omega_{2r}(0)\times ]t_0/2,t_0[$ may be covered by a collection of at most $L=L(n)$ cylinders 
$Q_r(x_l,t_0)$, $1\le l\le L$, we find 
\begin{equation*}%\label{2.7}
 \begin{split}
  \int_{Q_r(x_0,t_0/2)}|v|^pdz\le L\sup_{1\le l\le L}&\int_{Q_{r}(x_l,t_0)}|v|^pdz
  +Cr^{-2}\int_{\Omega_{2r}(0)\times ]0,t_0[}|v|^p\tau\, dxdt\\
  &+C\int_{\Omega_{2r}(0)\times ]0,t_0[}|\nabla v|^2|v|^{p-2}\tau\, dxdt.
  \end{split}
\end{equation*}
But by \eqref{2.2} we have $|v|^{p-2}\tau\le |v|^{p-2}t\le C\|f\|^{p-2}_{L^{2,\lambda}}$,
and from \eqref{2.7} we obtain
\begin{equation*}
 \begin{split}
  r^{-2}&\int_{\Omega_{2r}(0)\times ]0,t_0[}|v|^p\tau\, dxdt
  +\int_{\Omega_{2r}(0)\times ]0,t_0[}|\nabla v|^2|v|^{p-2}\tau\, dxdt\\
  &\le C\|f\|^{p-2}_{L^{2,\lambda}}\Big(r^{n-\lambda}\Psi(2r)
  +\int_{\Omega_{2r}(0)\times ]0,t_0[}|\nabla v|^2dxdt\Big)
  \le Cr^{n-\lambda}\|f\|^p_{L^{2,\lambda}}.
 \end{split}
\end{equation*}
Recalling that for each cylinder $Q_{r}(x_l,t_0)$, $1\le l\le L$, there holds \eqref{2.5},
we then obtain
\begin{equation*}
 \begin{split}
  \int_{Q_r(x_0,t_0/2)}|v|^pdz
  \le L\sup_{1\le l\le L}\int_{Q_{r}(x_l,t_0)}|v|^pdz
  +Cr^{n-\lambda}\|f\|^p_{L^{2,\lambda}} \le Cr^{n-\lambda}\|f\|^p_{L^{2,\lambda}},
  \end{split}
\end{equation*}
and \eqref{2.5} follows since $\lambda=\mu-2$.

Finally, for $t_0\le r^2$ and any $x_0\in\Omega$ equation \eqref{2.6} yields the gradient
bound 
\begin{equation*}
 \begin{split}
  \int_{Q_r(0,t_0)}|\nabla v|^2dz&\le\int_{\Omega_{2r}(0)}|f|^2\psi^2dx
   +16r^{-2}\int_{\Omega_{2r}(0)\times ]0,t_0[}|v|^2dxdt\\
  &\le Cr^{n-\lambda}\big(\|f\|^2_{L^{2,\lambda}}+\Psi(2r)\big)
  \le Cr^{n-\lambda}\|f\|^2_{L^{2,\lambda}}. 
  \end{split}
\end{equation*}
In view of \eqref{2.2} the same bound also holds for $t_0>r^2$ as can be seen by shifting time 
by $t_0-r^2$ and replacing $f$ with the function 
$\tilde{f}(x)=v(x,t_0-r^2)\in L^{2,\lambda}(\Omega)$. 
With $\lambda=\mu-2$ we obtain the bound 
$\|\nabla v\|_{L^{2,\mu}}\le C\|f\|_{L^{2,\lambda}}$,
as desired.

ii) Set $L_0:=\|f\|_{L^{2,\lambda}}$. As before, for any $x\in\Omega$ we have the bound
\begin{align*}
  |v(x,t)|&\le C\sum_{k=0}^\infty (2^k\sqrt{t})^{-\lambda/2}M_{\lambda/2,2^k\sqrt{t}}f(x).
\end{align*}
By assumption for $r=2^k\sqrt{t}\le R$ we can estimate 
\begin{equation*}
   M_{\lambda/2,r}(|f|)(x)\le\big(M_{\lambda,r}(|f|^2)(x)\big)^{1/2}\le\varepsilon_0,
\end{equation*}
whereas for any $r>0$ we have 
\begin{equation*}
 M_{\lambda/2,r}(|f|)(x)\le\big(M_{\lambda,r}(|f|^2)(x)\big)^{1/2}\le\|f\|_{L^{2,\lambda}}=L_0.
\end{equation*}
Let $k_0\in\N$ such that $2^{-k_0\lambda/2}L_0\le\varepsilon_0$.
Then for $0<t<T:=2^{-2k_0}R^2$ we find the uniform estimate
\begin{align*}
  |v(x,t)|& \le Ct^{-\lambda/4}\big(\sum_{k=0}^{k_0}2^{-k\lambda/2}\varepsilon_0
  + \sum_{k=k_0+1}^{\infty}2^{-k\lambda/2}L_0\big)
  \le Ct^{-\lambda/4}\varepsilon_0.
\end{align*}
Proceeding as in part i) of the proof, for any $0<t<T$, any $x_0\in\Omega$, 
and any $0<r<\sqrt{t/2}$ we then obtain the bound
\begin{equation*}%\label{2.4}
  \|v(t)\|^2_{L^2(\Omega_r(x_0))}\le Cr^nt^{-\lambda/2}\varepsilon_0^2
  \le Cr^{n-\lambda}\varepsilon_0^2;
\end{equation*}
similarly, we find
\begin{equation}\label{2.6a}
  \|v\|^p_{L^{p}(P_r(x_0,t_0))}\le Cr^{n+2}t^{-p\lambda/4}\varepsilon_0^p
  \le Cr^{n+2-\mu}\varepsilon_0^p
\end{equation}
whenever $0<2r^2<t_0<T$.
In order to derive the latter bound also for radii $r>0$ with $t_0/2\le r^2\le t_0\le T$
as in i) we may assume that $x_0=0$ and fix some numbers $0<t_0<T$, $r_0\ge\sqrt{t_0/2}$.
Setting
\begin{equation*}
     \Psi(r):=\sup_{0<t<t_0}r^{\lambda-n}\int_{B_{r}(0)}|v(t)|^2dx,\ r>0,
\end{equation*}
for $r=r_i=2^ir_0$, $i\in\N_0$, from \eqref{2.6} we obtain the bound
\begin{equation*}
 \begin{split}
     \Psi(r_i) &\le r_i^{\lambda-n}\int_{B_{2r_i}(0)}|f|^2dx
     +16C_1t_0r_i^{-2}\Psi(2r_i)\\
     & \le C_1M_{\lambda,r_{i+1}}(|f|^2)(0)+C_22^{-2i}\Psi(r_{i+1})
  \end{split}
\end{equation*}
for any $i\in\N$, with constants $C_1=2^{n-\lambda}$, $C_2=32C_1$ as before.

Suppose that $r_{i_0}\le R$ for some $i_0\in\N$.
Bounding $M_{\lambda,r_i}(|f|^2)(x)\le\varepsilon_0^2$ for $i\le i_0$ %such that $r_i\le R$
and $M_{\lambda,r_i}(|f|^2)(x)\le L_0^2$ else, %with $i_0=\max\{i; r_i\le R\}\ge 1$
by iteration we then obtain
\begin{equation*}
 \begin{split}
     \Psi(r_0) & \le C_1\varepsilon_0^2+C_2\Psi(r_1)
     \le C_1(1+C_2)\varepsilon_0^2+C_2^22^{-2}\Psi(r_2)\\
     &\le C_1(1+C_2+C_2^22^{-2})\varepsilon_0^2++C_2^22^{-2}C_22^{-4}\Psi(r_3)\le \dots\\
     &\le C_1\sum_{i=0}^{i_0-1}C_2^i2^{(1-i)i}\varepsilon_0^2 
     + C_1\sum_{i=i_0}^{k}C_2^i2^{(1-i)i}L_0^2
     +C_2^{k+1}2^{-k(k+1)}\Psi(r_{k+1}).
  \end{split}
\end{equation*}
Thus, if $i_0$ is such that $C_22^{1-i_0}\le(\varepsilon_0/L_0)^2\le 1/2$, that is, if 
$$\sqrt{2t_0}\le 2r_0=2^{1-i_0}r_{i_0}\le 2^{1-i_0}R\le C_2^{-1}(\varepsilon_0/L_0)^2 R,$$
upon passing to the limit $k\to\infty$ we obtain $\Psi(r_0)\le C\varepsilon_0^2$ and the 
analogue of \eqref{2.7} with $\varepsilon_0$ in place of $\|f\|_{L^{2,\lambda}}$.

Recalling the definition $T=2^{-2k_0}R^2$ with $k_0\in\N$ satisfying
$2^{-k_0\lambda/2}L_0\le\varepsilon_0$, we see that these bounds hold true for 
$$0<t_0/2\le r_0^2\le t_0\le
T_0:=R^2\cdot\min\{(\varepsilon_0/L_0)^{4/\lambda},C_2^{-2}(\varepsilon_0/L_0)^4\}.$$
Using \eqref{2.6a}, the remainder of the proof of \eqref{2.5} in part i) now may be copied 
unchanged to yield the claim.
\end{proof}

The assertions of Theorems \ref{thm1.1} and \ref{thm1.2} now are a consequence of the 
following result.
 
\begin{lemma} \label{lem2.2}
i) For any $p>2^*$ there exists a constant $\varepsilon_0>0$ such that for any 
$u_0\in L^{2,\lambda}(\Omega)$ with $\|u_0\|_{L^{2,\lambda}}\le\varepsilon_0$ there exists a 
unique solution $u\in L^{p,\mu}(\Omega\times]0,\infty[)$
to the Cauchy problem \eqref{1.1} %with $\nabla u\in L^{2,\mu}(\Omega\times]0,\infty[)$ and 
such that 
\begin{equation}\label{2.8}
  \|u\|_{L^{p,\mu}}\le C\|u_0\|_{L^{2,\lambda}}.
\end{equation}
ii) Let $u_0\in L^{2,\lambda}(\Omega)$ and suppose that there exists a number $R>0$ such that
\begin{equation*}
  \sup_{x_0\in\Omega,\;0<r<R}r^{\lambda-n}\int_{\Omega_r(x_0)}|u_0|^2 dx
  \le\varepsilon_0^2,
\end{equation*}
where $\varepsilon_0>0$ is as determined in i).
Then there exists a unique smooth solution $u$ to \eqref{1.1} on an interval $]0,T_0[$,
where $T_0/R^2=C(\varepsilon_0^{-1}\|u_0\|_{L^{2,\lambda}(\Omega)})>0$.
\end{lemma}

\begin{proof}
For $u_0\in L^{2,\lambda}(\R^n)$ set $w_0=S_{\Omega}u_0$. For suitable $a>0$ let 
\begin{equation*}
    X:=\{v\in L^{p,\mu}(\Omega\times]0,T_0[);\; \|v\|_{L^{p,\mu}}\le a\},
\end{equation*}
where $T_0>0$ in the case of the assumptions in i) may be chosen arbitrarily large and 
otherwise is as in assertion ii) of Lemma \ref{lem2.1}.

Then $X$ is a closed subset of the Banach space $L^{p,\mu}=L^{p,\mu}(\Omega\times]0,T_0[)$. 
Moreover, for any $v\in X$ we have $|v|^{p-2}v\in L^{p/(p-1),\mu}$. 
By Lemma 4.1 in \cite{Blatt-Struwe-2013}
there exists a unique solution $w=S(v|v|^{p-2})\in L^{p,\mu}$
of the Cauchy problem
\begin{equation*}
  w_t -\Delta w = |v|^{p-2}v\hbox{ on }\Omega\times]0,T_0[,\ w\big|_{t=0}=0,
\end{equation*}
with 
\begin{equation*}
  \|w\|_{L^{p,\mu}}\le C\|v\|^{p-1}_{L^{p,\mu}}\le Ca^{p-1}.
\end{equation*}
For sufficiently small $\varepsilon_0, a>0$ then the map 
\begin{equation*}
   \Phi\colon X\ni v\mapsto w_0+w\in X,
\end{equation*} 
and for $v_{1,2}\in X$ with corresponding $w_i=S(v_i|v_i|^{p-2})$, $i=1,2$, we can estimate
\begin{equation*}
  \begin{split}
  \|\Phi(v_1)&-\Phi(v_2)\|_{L^{p,\mu}} 
  =\|w_1-w_2\|_{L^{p,\mu}}
  \le C\|v_1|v_1|^{p-2}-v_2|v_2|^{p-2}\|_{L^{p/(p-1),\mu}}\; .
  \end{split}
\end{equation*}
The latter can be bounded 
\begin{equation*}
  \begin{split}
  \|v_1|v_1|^{p-2}&-v_2|v_2|^{p-2}\|_{L^{p/(p-1),\mu}}
  \le C\big(\|v_1\|^{p-2}_{L^{p,\mu}}
  +\|v_2\|^{p-2}_{L^{p,\mu}}\big)
  \|v_1-v_2\|_{L^{p,\mu}}.
  \end{split}
\end{equation*}
Thus for sufficiently small $a>0$ we find 
\begin{equation*}
  \begin{split}
  \|\Phi(v_1)&-\Phi(v_2)\|_{L^{p,\mu}}
  \le Ca^{p-2}\|v_1-v_2\|_{L^{p,\mu}}\le\frac12\|v_1-v_2\|_{L^{p,\mu}}.
  \end{split}
\end{equation*}
By Banach's theorem the map $\Phi$ has a unique fixed point $u\in X$, 
and $u$ solves the initial value problem \eqref{1.1} in the
sense of distributions. %with $\nabla u\in L^{2,\mu}(\Omega\times]0,\infty[)$.
Finally, for sufficiently small $a,\varepsilon_0>0$ we can
invoke Proposition 4.1 in \cite{Blatt-Struwe-2013} to show that $u$, in fact, is a smooth 
global solution of \eqref{1.1}.
\end{proof}

\section{Ill-posedness for ``large'' data}
\subsection{Minimal solutions for non-negative initial data}
In order to obtain a notion of solution of \eqref{1.1} on $\Omega\times ]0,\infty[$ for 
arbitrary nonnegative initial data $u_0\ge 0$, following Baras-Cohen \cite{Baras-Cohen-1987}
for $n\in\N$ we solve the initial value problem
\begin{equation}\label{4.1}
  u_{n,t} -\Delta u_n = f_n(u_n) = \min\{u_n^{p-1},n^{p-1}\}
  \hbox{ on } \Omega\times ]0,\infty[,\ u=0 \hbox{ on } \partial\Omega\times ]0,\infty[,
\end{equation} 
with initial data 
\begin{equation}\label{4.2} 
  u_n(x,0) = u_{0n}(x):=\min\{u_0(x),n\}\ge 0.
\end{equation}
As the right-hand side $f_n(u_n)$ in \eqref{4.1} is uniformly bounded, for any $n\in\N$ there 
exists a unique global solution of \eqref{4.1}, \eqref{4.2}. 
By the maximum principle, positivity of the initial 
data is preserved and $u_n$ is monotonically increasing in $n$. 
Hence, the pointwise limit $u(x,t) := \lim_{n \rightarrow \infty} u_n (x,t)\le\infty$ exists. 
Inspired by Baras and Cohen \cite{Baras-Cohen-1987} we call this limit 
the {\em minimal solution} of problem \eqref{1.1} for the given data $u_0$. Moreover,
similar to their Proposition 2.1 we have $u\le v$ for any $v$ which is 
an {\em integral solution} $v$ of \eqref{1.1} in the sense that 
\begin{equation} \label{4.3}
 v(t) = S_tu_0 + \int_0^t S_{t-s}v^{p-1}(s)ds,
\end{equation}
where for brevity we now write $(S_t)_{t\ge 0}$ for the heat semigroup on $\Omega$, defined by
$$
 S_tw(x)=\int_{\Omega}\Gamma(x,y,t)w(y) dy,
$$
with $\Gamma > 0$ denoting the fundamental solution of the heat equation on $\Omega$.

Indeed, by Duhamel's principle the $u_n$ satisfy the integral equation
\begin{equation}\label{4.4} 
 u_n(t) = S_t u_{0n}+\int_0^t S_{t-s}f_n(u_n(s)) ds.
\end{equation}
Recalling that the sequence $u_n$ is monotonically increasing in $n$, from Beppo-Levi's theorem
on monotone convergence we find that $u$ satisfies \eqref{4.3}. On the other hand, for each $n$ 
and any integral solution $v$ of \eqref{1.1} clearly there holds $u_n\le v$.

With these prerequisites we now show that there are initial data $u_0\in L^{p,\mu}(\Omega)$  
with even $\nabla u_0 \in L^{2,\mu}$ such that the minimal solution $u$ to \eqref{1.1} 
satisfies $u\equiv\infty$ on $\Omega\times]0,\infty[$, that is, undergoes {\em complete 
instantaneous blow-up}. The following arguments are modelled on corresponding results on 
complete instantaneous blow-up by Galaktionov and Vazquez \cite{Galaktionov-Vazquez-1997}
in the case when $\Omega = \R^n$.

\subsection{Complete instantaneous blow-up}
It is well-known that on a bounded domain $\Omega$ equation \eqref{1.1} may be interpreted 
as the negative gradient flow of the energy
\begin{equation*}
 E(u)= E_{\Omega}(u)=\int_{\Omega}\big(\frac12|\nabla u|^2-\frac1p|u|^p\big)dx.
\end{equation*}
As observed by Ball \cite{Ball-1978}, Theorem 3.2, sharpening an earlier result of
Kaplan \cite{Kaplan-1963}, for data $u_0$ with $E(u_0)<0$ the solution
to \eqref{1.1} blows up in finite time. Indeed, Ball \cite{Ball-1978}, Theorem 3.2, observes that
testing the equation \eqref{1.1} with $u$ leads to the differential inequality
\begin{equation*}
 \begin{split}
   \frac12\frac{d}{dt}\|u(t)\|^2_{L^2}& =-\int_{\Omega\times\{t\}}\big(|\nabla u|^2-|u|^p\big)dx
   = - 2E(u(t))+ \frac{p-2}{p}\|u(t)\|^p_{L^p}\\
   & \ge -2E(u_0)+c_0\|u(t)\|^p_{L^2}\ge c_0\|u(t)\|^p_{L^2}
 \end{split}
\end{equation*} 
for some constant $c_0>0$. Hence we find
$$\|u(t)\|_{L^2}\ge \big(\|u_0\|^{(2-p)/2}_{L^2}-c_0(p-2)t\big)^{-2/(p-2)},$$
and $u(t)$ must blow up at the latest at time $T= c_0^{-1}(p-2)^{-1} \|u_0\|^{(2-p)/2}_{L^2}$.

In order to obtain data $u_0\in L^{p,\mu}$ leading to
instantaneous complete blow-up, we combine this observation with the following well-known 
scaling property of equation \eqref{1.1}: Whenever $u$ is a solution 
of \eqref{1.1} on $\Omega$, then for any $R>0$, any $x_0\in\R^n$ the function
\begin{equation}\label{4.5} 
 u_{R,x_0}(x,t) = R^{-\alpha}u(R^{-1}(x-x_0), R^{-2}t)
\end{equation}
with $\alpha= \frac{2}{p-2}$
is a solution of \eqref{1.1} on the scaled domain  
$$\Omega_{R,x_0}:=\{x\in\R^n;\; R^{-1}(x-x_0)\in\Omega\}.$$
Clearly we may assume that $0\in\Omega$.

\begin{theorem} \label{thm4.1}
Let $0\le w_0\in C_c^\infty(B_1(0))$ with $E_{B_1(0)}(w_0)<0$. 
Set
 $$
  M = M_{w_0} = \sup_{|y|\le 1}\big(|y|^{\alpha}w_0(y)\big), %\sup_{y \in B_1(0)} |y|^\alpha |w_0(y)|, 
 $$
where $\alpha= \frac{2}{p-2}$ as above.  
Then for every initial data $0\le u_0\in C^0(\Omega\setminus\{0\})$ satisfying
 $$
  \liminf_{x\to 0}\big(u_0(x)-M |x|^{-\alpha}\big) > 0
 $$
the minimal solution $u$ to \eqref{1.1} blows up completely instantaneously.
\end{theorem}

\begin{proof}
By Ball's above result, the solution $w$ to \eqref{1.1} on $B_1(0)\times ]0,T[$
with initial data $w(0)=w_0$ blows up after some finite time $T$ at a point $y_0$.

Fix $R_0>0$ with $B_{R_0}(0)\subset\Omega$ and such that  
$$u_0(x)> M|x|^{-\alpha} \hbox{ for }|x|\le R_0.$$
For $R < R_0$ and $x_0\in\Omega$ with $|x_0|\le R_0-R$ consider the rescaled solutions 
$$w_{R,x_0}(x,t):= R^{-\alpha} w(R^{-1}(x-x_0), R^{-2}t)$$
on $B_R(x_0)\times [0,R^2T[$ that blow up at time $R^2T$.

Since by assumption we have
\begin{equation*}
  w_{R,0}(x,0)=R^{-\alpha}w_0(R^{-1}x)
  \le M|x|^{-\alpha}<u_0(x)\text{ on } B_R(0),
\end{equation*}
by continuity of $u_0$ away from $x=0$ and continuity of $w_0$ there is a number 
$\delta=\delta(R)>0$ such that 
\begin{equation*}
 w_{R,x_0}(x,0) < u_0(x) \text{ on } B_R(x_0)
\end{equation*}
for all $x_0$ with $|x_0|<\delta$. Since in addition
$u\ge 0=w_{R,x_0}$ on $\partial B_R(x_0)\times [0,r^2T[$,
by the maximum principle for any $\varepsilon >0$, any 
$n\ge\|w_{R,x_0}\|_{L^{\infty}(B_R(x_0)\times [0,R^2T-\varepsilon])}$ there holds
\begin{equation*}
 u(x,t)\ge u_n(x,t)\ge w_{R,x_0}(x,t)\text { on } B_R(x_0)\times [0,R^2T-\varepsilon],
\end{equation*}
where $u_n$ solves \eqref{4.1} for each $n\in\N$. 
Passing to the limit $\varepsilon \to 0$, we then find
\begin{equation*}
 \begin{split}
  u(x_0+Ry_0,R^2T)
  &= \Big(S_{R^2T}u_0+\int_0^{R^2T}S_{R^2T-s}f(u(s))ds\Big)(x_0+Ry_0)\\
  &= \lim_{n\to\infty}\Big(S_{R^2T}u_{0n}+\int_0^{R^2T}S_{R^2T-s}f_n(u_n(s))ds\Big)(x_0+Ry_0)\\
  &\ge\lim_{t\uparrow R^2T}w_{R,x_0}(x_0+Ry_0,t)=\infty
 \end{split}
\end{equation*}
for all $x_0\in B_{\delta}(0)$. 

Since $R>0$ may be chosen arbitrarily small, we conclude that for any sufficiently small 
$t>0$ there holds $\mathcal L^n(\{x\in\Omega; u(x,t)=\infty\})>0$.
But then positivity of $\Gamma$ and Duhamel's principle \eqref{4.3} yield
\begin{equation*}
  u(x,t) = \Big(S_tu_0 + \int_0^t S_{t-s}u^{p-1}(s)ds\Big)(x)=\infty.
\end{equation*}
for any $t>0$ and any $x\in\Omega$.
\end{proof}

\subsection{Open problems}
Can data $u(T)$ that lead to instantaneous complete blow-up arise from solutions of \eqref{1.1}
with bounded energy? What is the smallest number $M>0$ so that the conclusion of 
Theorem \ref{thm4.1} holds true? Can one show that at least for exponents $p$ strictly less
than the Joseph-Lundgren \cite{Joseph-Lundgren-1972} exponent
\begin{equation*}
 p_{JL} = 2 + \frac {4}{n-4-2\sqrt{n-1}}\ \hbox{ if } n\ge 11,
 \ p_{JL} = \infty\ \hbox{ if }n\le 10,
\end{equation*}
we have $M=\alpha(n-2-\alpha)=:c_*$, 
where $c_*$ appears as coefficient in the singular solution 
$u_*(x):= c_*|x|^{-\alpha}$ of the time-independent equation \eqref{1.1} on $\R^n$?
(The significance of the exponent $p_{JL}$ is illustrated for instance in Lemma 9.3 of 
\cite{Galaktionov-Vazquez-1997}.)

Hopefully, we will be able to answer some of these questions in the future.

\end{document}